\documentclass[11 pt]{amsart}

\usepackage{amsfonts}
\usepackage{amssymb}
\usepackage{amsmath}
\usepackage{latexsym}
\usepackage{mathrsfs}
\usepackage{amsthm}
\usepackage{verbatim}

\newtheorem{thrm}{Theorem}[section]
\newtheorem{lem}[thrm]{Lemma}
\newtheorem{cor}[thrm]{Corollary}
\newtheorem{prop}[thrm]{Proposition}

\theoremstyle{definition}
\newtheorem{defn}[thrm]{Definition}
\newtheorem{exmple}[thrm]{Example}
\newtheorem{rmk}[thrm]{Remark}
\newtheorem{conj}[thrm]{Conjecture}
\newtheorem{ques}[thrm]{Question}

\newtheorem{constr}[thrm]{Construction}
\newtheorem{conv}[thrm]{Convention}

\begin{document}

\newcommand{\vol}{\mathrm{vol}}

\newcommand{\ord}{\mathrm{ord}}

\newcommand{\Supp}{\mathrm{Supp}}

\newcommand{\Sing}{\mathrm{Sing}}

\newcommand{\Val}{\mathrm{Val}}

\newcommand{\cent}{\mathrm{cent}}

\newcommand{\lct}{\mathrm{lct}}

\title{Algebraic bounds on analytic multiplier ideals}
\author{Brian Lehmann}
\thanks{The author is supported by NSF Award 1004363.}
\address{Department of Mathematics, Rice University \\
Houston, TX \, \, 77005}
\email{blehmann@rice.edu}

\begin{abstract}
Given a pseudo-effective divisor $L$ we construct the diminished ideal $\mathcal{J}_{\sigma}(L)$, a ``continuous'' extension of the asymptotic multiplier ideal for big divisors to the pseudo-effective boundary.  Our main theorem shows that for most pseudo-effective divisors $L$ the multiplier ideal $\mathcal{J}(h_{min})$ of the metric of minimal singularities on $\mathcal{O}_{X}(L)$ is contained in $\mathcal{J}_{\sigma}(L)$.  We also characterize abundant divisors using the diminished ideal, indicating that the geometric and analytic information should coincide.
%
%
%
%
\end{abstract}

\maketitle

\section{Introduction}

Suppose that $X$ is a smooth projective complex variety and that $\mathcal{L}$ is a holomorphic line bundle on $X$.  A classical theorem of Kodaira states that if $\mathcal{L}$ carries a smooth hermitian metric with positive curvature then $\mathcal{L}$ is ample.  More generally, the positivity of singular hermitian metrics carried by $\mathcal{L}$ should be related to its ``algebraic positivity'', i.e.~the existence of sections of tensor powers of $\mathcal{L}$.  Among all positive singular metrics carried by a line bundle $\mathcal{L}$, there is a class of metrics, denoted $h_{min}$, which have the ``minimal'' possible singularities and the closest ties to the geometry of $\mathcal{L}$.

The main tool for relating metrics to geometric properties is the multiplier ideal.  These ideals have played a key role in recent work in birational geometry and have thus become important objects of study in their own right.  Our goal is to give bounds on the behavior of the multiplier ideal $\mathcal{J}(h_{min})$ using algebraically defined multiplier ideals.  This question seems to have first appeared in \cite{del00} where it is conjectured that when $\mathcal{L}$ is a big line bundle $\mathcal{J}(h_{min})$ coincides with the asymptotic multiplier ideal.  We focus on the case when $\mathcal{L}$ is pseudo-effective but not big.  Since the vanishing theory for big divisors is well-established, this case is more interesting for geometric applications.  However, the situation is much more subtle; even for surfaces there is no conjectural algebraic description of $\mathcal{J}(h_{min})$.  Since one can find lower bounds for $\mathcal{J}(h_{min})$ using sections of $\mathcal{L}$, the main point is to construct an interesting upper bound.

We introduce the diminished ideal $\mathcal{J}_{\sigma}(\mathcal{L})$ in Section \ref{diminishedsection}.  This ideal can be thought of as a ``continuous'' extension of the asymptotic multiplier ideal on the big cone to the pseudo-effective boundary.  In contrast to previous constructions for pseudo-effective divisors, the diminished ideal retains many of the desirable properties of the asymptotic multiplier ideal for big divisors.  It plays an important role in the numerical theory of divisors; in particular, Theorem \ref{abundantequality} shows that diminished ideals can be used to detect whether a divisor is abundant, which is a key problem in the minimal model program.

The natural analogue of the conjecture of \cite{del00} is:

\begin{conj} \label{psefequality}
Let $X$ be a smooth variety and let $L$ be a pseudo-effective $\mathbb{R}$-divisor on $X$.  Then
\begin{equation*}
\mathcal{J}(T_{min}) \subset \mathcal{J}_{\sigma}(L)
\end{equation*}
where $T_{min}$ is a current of minimal singularities in the numerical class of $L$.
\end{conj}

In order to give the most precise statement, we have passed from line bundles $\mathcal{L}$ to $\mathbb{R}$-divisors $L$ and from metrics $h_{min}$ to currents $T_{min}$.
In contrast to the case when $L$ is big, there are examples where the containment in Conjecture \ref{psefequality} is strict.

Our main theorem shows that Conjecture \ref{psefequality} holds for ``most'' pseudo-effective divisors $L$.  This generalizes \cite{del00} Theorem 1.11, which shows that $\mathcal{J}(h_{min}) = \mathcal{J}(\Vert L \Vert)$ for ``most'' big divisors.

\begin{thrm} \label{mostlypsefcontainment}
Let $X$ be a smooth variety and let $L$ be a pseudo-effective $\mathbb{R}$-divisor on $X$.
\begin{enumerate}
\item There is an open dense subset $U$ of $\mathbb{R}_{>0}$ such that for every $\alpha \in U$ we have
\begin{equation*}
\mathcal{J}(T_{min, \alpha L}) \subset \mathcal{J}_{\sigma}(\alpha L)
\end{equation*}
where $T_{min, \alpha L}$ is a current of minimal singularities in the numerical class of $\alpha L$.
\item Suppose that $L$ has a Zariski decomposition (that is, there is some birational map $\phi: Y \to X$ from a smooth variety $Y$ such that $P_{\sigma}(\phi^{*}L)$ is nef).  Then
\begin{equation*}
\mathcal{J}(T_{min}) \subset \mathcal{J}_{\sigma}(L)
\end{equation*}
where $T_{min}$ is a current of minimal singularities in the numerical class of $L$.
\end{enumerate}
\end{thrm}

Combined with Theorem \ref{abundantequality} below, Theorem \ref{mostlypsefcontainment} determines completely $\mathcal{J}(h_{min})$ for varieties with good birational behavior, such as smooth Fano and toric varieties and minimal surfaces $S$ with $\kappa(S) = 0$.

The remaining cases of Conjecture \ref{psefequality} would be settled by the multiplier ideal version of the Openness Conjecture of \cite{dk01}.  This conjecture predicts that if $\varphi$ is a plurisubharmonic function then $\mathcal{J}(\varphi) = \mathcal{J}((1+\epsilon)\varphi)$ for sufficiently small $\epsilon >0$.  In the setting of Theorem \ref{mostlypsefcontainment} the conjecture yields for any pseudo-effective $\mathbb{R}$-divisor $L$
\begin{equation*}
\mathcal{J}(T_{min}) = \mathcal{J}((1+\epsilon)T_{min}) \subset \mathcal{J}_{\sigma}((1+\epsilon)L) = \mathcal{J}_{\sigma}(L)
\end{equation*}
where the second and third containments follow from Theorem \ref{mostlypsefcontainment} (1) and the remarks after Definition \ref{sigmadefn} respectively.

Our approach to Theorem \ref{mostlypsefcontainment} is essentially algebraic in nature (see Remark \ref{otremark}).  The proof involves the study of the behavior of asymptotic multiplier ideals under perturbations of the divisor.  The connection to analysis is then made using valuation theory and the results of \cite{boucksom04} on the divisorial Zariski decomposition.

There is a variation of Theorem \ref{mostlypsefcontainment} that holds with no restrictions on $L$.  Following \cite{hacon04}, for a pseudo-effective divisor $L$ we define the perturbed ideal $\mathcal{J}_{-}(L)$ as the asymptotic multiplier ideal of $L$ increased by a sufficiently small ample divisor (see Definition \ref{perturbedidealdefn}).  Since $\mathcal{J}_{\sigma}(L) \subset \mathcal{J}_{-}(L)$, the following theorem gives a weaker version of Conjecture \ref{psefequality}.

\begin{thrm} \label{idealcomparison}
Let $X$ be a smooth projective variety and let $L$ be a pseudo-effective Cartier divisor on $X$.  Then
\begin{equation*}
\mathcal{J}(T_{min}) \subset \mathcal{J}_{-}(L)
\end{equation*}
where $T_{min}$ is a current of minimal singularities in the numerical class of $L$.
\end{thrm}

\subsection{Abundant Divisors}

It is interesting to ask when the containment of Conjecture \ref{psefequality} is actually an equality.  There are examples showing that this containment may be strict.  Nevertheless, in special geometric situations one can say more.  Recall that a divisor $L$ is called abundant if its Iitaka dimension and numerical dimension coincide (see Definition \ref{abundancedefn}).  Generalizing \cite{russo09}, we show that abundance can be detected using multiplier ideals.

\begin{thrm} \label{abundantequality}
Let $X$ be a smooth variety and let $L$ be a divisor on $X$ with $\kappa(L) \geq 0$.  Then $L$ is abundant iff
\begin{equation*}
\mathcal{J}(\Vert m\phi^{*}L \Vert) = \mathcal{J}_{\sigma}(m\phi^{*}L)
\end{equation*}
for every positive integer $m$ and every birational map $\phi: Y \to X$ from a smooth variety $Y$.
\end{thrm}

Note that for an abundant divisor $\mathcal{J}(T_{min})$ should coincide with both of these ideals by Conjecture \ref{psefequality}.  One wonders if there are any other natural geometric conditions one can impose on $L$ that would give an equality in Conjecture \ref{psefequality}.    

\subsection{Outline}

The basic observation behind Theorem \ref{mostlypsefcontainment} is as follows.  The main tool for understanding the singularities of a current is Demailly approximation \cite{demailly93}.  However, this only allows us to find the multiplier ideal of a perturbation of our current.  The obstruction to understanding our original ideal is the difficult Openness Conjecture of \cite{dk01}.

In contrast, it is much easier to understand the behavior of asymptotic multiplier ideals under perturbations of the divisor.  We study this behavior in Section \ref{perturbedsection} using ideas of \cite{hacon04} and the relationship between valuations and asymptotic multiplier ideals described in \cite{jm10}.

We then use valuation theory to relate the algebraic perturbations to the analytic multiplier ideal.  Although the connections between analytic multiplier ideals and valuations are not fully understood, we only need the results of \cite{boucksom04} describing the Lelong numbers of $T_{min}$ using asymptotic divisorial valuations.  This allows us to give a bound on $\mathcal{J}(T_{min})$ whenever the perturbed ideal $\mathcal{J}_{\sigma}(L)$ is determined by its divisorial valuations.  The requirements on $L$ in Theorem \ref{mostlypsefcontainment} are precisely those which guarantee this condition on $\mathcal{J}_{\sigma}(L)$.


\begin{rmk} \label{otremark}
Rather than using valuation theory to prove Theorem \ref{mostlypsefcontainment}, one could try to use Ohsawa-Takegoshi directly.  Emulating the proof of \cite{del00} Theorem 1.11, one can prove that for a pseudo-effective divisor $L$
\begin{equation*}
\mathcal{J}_{+}(T_{min}) \subset \mathcal{J}_{-}(L).
\end{equation*}
where $\mathcal{J}_{+}$ denotes the ``upper-regularized'' multiplier ideal.
Although this approach will lead to a weaker version of Theorem \ref{mostlypsefcontainment}, it  does not seem to yield the more precise results needed to analyze the examples in Section \ref{examplessection}.
\end{rmk}

\subsection{Acknowledgements}

This paper was initially a joint project with E.~Eisenstein.  I would like to thank him for his insights and for his encouragement.  Thanks also to M.~Jonsson and to M.~Musta\c t\u a for numerous helpful conversations.

\section{Background} \label{backgroundsection}

Throughout $X$ will be a smooth projective variety over $\mathbb{C}$.  The term `divisor' will always refer to an $\mathbb{R}$-Cartier divisor (equivalently, an $\mathbb{R}$-Weil divisor) unless otherwise qualified.  We will use the standard notations $\sim$, $\sim_{\mathbb{Q}}$, $\sim_{\mathbb{R}}$, and $\equiv$ to denote respectively linear, $\mathbb{Q}$-linear, $\mathbb{R}$-linear, and numerical equivalence of divisors.  The numerical class of an $\mathbb{R}$-divisor $L$ will be denoted by $c_{1}(L)$.

\subsection{$\mathbb{R}$-divisors}

Suppose that $L = \sum_{i} a_{i}L_{i}$ is an $\mathbb{R}$-Weil divisor.  The round-down of $L$ is $\lfloor L \rfloor = \sum_{i} \lfloor a_{i} \rfloor L_{i}$.  We denote the fractional part of $L$ by $\{ L \} := L - \lfloor L \rfloor$.  

Recall that two $\mathbb{R}$-divisors $L,L'$ are linearly equivalent if $\{ L \} = \{ L' \}$ and
$\lfloor L \rfloor - \lfloor L' \rfloor = \mathrm{div}(f)$ for some rational function $f$ on $X$.
The linear system $|L|$ associated to $L$ is the set of all effective divisors linearly equivalent to $L$.  Note that if $L$ and $L'$ are $\mathbb{R}$-divisors with $L \sim L'$ then
\begin{itemize}
\item $kL \sim kL'$ for every positive integer $k$, and
\item for any effective $\mathbb{R}$-divisor $D$ we have $L+D \sim L' + D$.
\end{itemize}
An $\mathbb{R}$-divisor $A$ is said to be ample if it is a sum $A = \sum c_{i} A_{i}$ where $c_{i}>0$ and $A_{i}$ is an ample Cartier divisor.  \cite{lazarsfeld04} Proposition 1.3.13 shows that if $A$ is an ample $\mathbb{R}$-divisor and $A' \equiv A$ then $A'$ is also ample.

Suppose that some integer multiple of a divisor $L$ is linearly equivalent to an effective divisor.  \cite{nakayama04} shows that one can associate an Iitaka fibration $\psi$ to $L$, characterized up to birational equivalence by the fact that there is a unique effective divisor $\mathbb{Q}$-linearly equivalent to $L|_{F}$ for a very general fiber $F$ of $\psi$.  The Iitaka dimension $\kappa(L)$ is the dimension of the image of the Iitaka fibration.  If no multiple of $L$ is linearly equivalent to an effective divisor, we set $\kappa(L) = -\infty$.

The $\mathbb{R}$-stable base locus $\mathbf{B}_{\mathbb{R}}(L)$ is defined to be
\begin{equation*}
\mathbf{B}_{\mathbb{R}}(L) := \bigcap \{ \; \Supp(C) \; | \; C \geq 0 \; \; \textrm{and} \; \; C \sim_{\mathbb{R}} L \}.
\end{equation*}
This is always a Zariski-closed subset of $X$; we do not associate any scheme structure to it.  When $L$ is a $\mathbb{Q}$-divisor, this definition coincides with the usual stable base locus.

By perturbing $L$ by a small ample divisor, we can define a better-behaved variant of the $\mathbb{R}$-stable base locus.

\begin{defn}[\cite{elmnp05}, Lemma 1.14]
Let $X$ be smooth and let $L$ be a pseudo-effective divisor on $X$.  We define the diminished base locus to be
\begin{equation*}
\mathbf{B}_{-}(L) = \bigcup_{A} \mathbf{B}_{\mathbb{R}}(L + A)
\end{equation*}
as we vary over all ample $\mathbb{R}$-divisors $A$.
\end{defn}

It is checked in \cite{elmnp05} that $\mathbf{B}_{-}(L) \subset \mathbf{B}_{\mathbb{R}}(L)$ and that the diminished base locus only depends on the numerical class of $L$.   In some sense the diminished base locus measures the failure of $L$ to be nef.

\subsection{Valuation Theory}

In this section we recall some basic facts about valuations.  We closely follow the conventions of \cite{jm10} and refer there for more details.

We use $\Val(X)$ to denote the set of real valuations of the function field $K(X)$ that have center on $X$.  Recall that the center of a valuation $v$ (if it exists) is the unique subvariety $V$ admitting a local inclusion of local rings $\mathcal{O}_{X,V} \hookrightarrow \mathcal{O}_{v}$.  We denote the center of a valuation $v$ by $\cent(v)$.  The trivial valuation is the valuation with center at the generic point of $X$; we will use the notation $\Val^{*}(X)$ to denote the subset of non-trivial valuations.  $\Val(X)$ carries the natural weak topology induced by the evaluation maps $\mathrm{eval_{f}}: \Val(X) \to \mathbb{R}$ where $f$ varies over all nonzero rational functions on $X$.  

For a valuation $v \in \Val(X)$ and an ideal sheaf $\mathfrak{a}$ we define
\begin{equation*}
v(\mathfrak{a}) = \min\{ v(f) | f \in \mathfrak{a}(U) \textrm{ for an open set } U \textrm{ intersecting } \cent(v) \}.
\end{equation*}
We also extend $v$ to $\mathbb{R}$-divisors by $\mathbb{R}$-linearity.

A valuation is divisorial if it is proportional to the valuation measuring the order of vanishing along some divisor $E$ on a birational model of $X$.  Divisorial valuations are dense in $\Val(X)$ and will play a key role in our analysis.

A valuation $v$ is quasi-monomial if there is a birational map $\phi: Y \to X$ from a nonsingular variety $Y$, a point $\eta \in Y$, and a system of coordinates $(y_{1},\ldots,y_{r})$ at $\eta$ such that $v$ has the following description.  For any $f \in \mathcal{O}_{Y,\eta}$ write the image $\widehat{f} \in \widehat{\mathcal{O}_{Y,\eta}}$ as
\begin{equation*}
\widehat{f} = \sum_{(i_{1},\ldots,i_{r}) \in \mathbb{Z}_{\geq 0}^{r}} c_{i_{1},\ldots,i_{r}} y_{1}^{i_{1}} \ldots y_{r}^{i_{r}}
\end{equation*}
where each $c_{i_{1},\ldots,i_{r}}$ is either zero or a unit.  Then there should be a vector $(\alpha_{1},\ldots,\alpha_{r}) \in \mathbb{R}_{\geq 0}^{r}$ such that
\begin{equation*}
v(f) = \min \left\{ \left. \sum_{j=1}^{r} \alpha_{j} i_{j} \right| c_{i_{1},\ldots,i_{r}} \neq 0 \right\}.
\end{equation*}

Let $(Y,D)$ be a pair consisting of a nonsingular variety $Y$, a reduced effective simple normal crossing divisor $D$, and a birational morphism $\pi: Y \to X$ that is an isomorphism outside of $\Supp(D)$.  Define $QM(Y,D)$ to be the set of all quasi-monomial valuations that can be described at a point $\eta \in Y$ using a coordinate system $(y_{1},\ldots,y_{r})$ such that each $y_{i}$ defines a component of $D$ at $\eta$.  By \cite{jm10} Remark 3.4 every quasi-monomial valuation lies in some $QM(Y,D)$.

$QM(Y,D)$ is naturally a subset of $\Val_{X}$ and thus carries the subspace topology.  For each $(Y,D)$, \cite{jm10} Section 4.3 constructs a retraction map $r_{Y,D}: \Val_{X} \to QM(Y,D)$ taking a valuation $v$ to the unique quasi-monomial valuation $w \in QM(Y,D)$ such that $w(D_{i}) = v(D_{i})$ for every irreducible component $D_{i}$ of $D$.

We next give an informal introduction to the log discrepancy function $A: \Val(X) \to \mathbb{R}_{\geq 0} \cup \{ \infty \}$, referring to \cite{jm10} for details.  For a valuation $v$ measuring the order of vanishing along a divisor $E$ over $X$, the log discrepancy is
\begin{equation*}
A(v) = v(K_{Y/X})+1
\end{equation*}
where $Y$ is some smooth birational model extracting $E$.  The log discrepancy admits a similar natural description for any quasi-monomial valuation.  For an arbitrary valuation, we define $A(v)$ by taking a supremum over the discrepancies of the quasi-monomial valuations $r_{Y,D}(v) \in QM(Y,D)$ as $Y$ and $D$ vary.  Thus:

\begin{lem}[\cite{jm10}, Lemma 5.7]
The log discrepancy function $$A: \Val(X) \to \mathbb{R}_{\geq 0} \cup \{ \infty \}$$ is lower semi-continuous.
\end{lem}

We will need the following consequence of \cite{jm10} Theorem 4.9.

\begin{lem} \label{divdensity}
Let $X$ be a smooth variety and let $v \in \Val(X)$.  There is a net of divisorial valuations $w_{\alpha}$ converging to $v$ such that
\begin{equation*}
A(v) = \lim_{\to} A(w_{\alpha})
\end{equation*}
\end{lem}

\begin{proof}
Assume first that $A(v) < \infty$.  It suffices to show that for any neighborhood $U$ of $v$ and any $\epsilon > 0$ there is a divisorial valuation $w_{U,\epsilon} \in U$ such that
\begin{equation*}
A(v) - \epsilon <A(w_{U,\epsilon}).
\end{equation*}
Choose some model $(Y,D)$ such that $A(v)-\epsilon < A(r_{Y,D}(v))$. 
Since log discrepancies do not decrease under passing to higher models, by \cite{jm10} Theorem 4.9 we may also assume that $r_{Y,D}(v) \in U$.

Recall that $A(v)$ is continuous and that divisorial valuations are dense on each $QM(Y,D)$.  Thus, we may choose $w_{U,\epsilon}$ sufficiently close to $r_{Y,D}(v)$ to satisfy all the desired criteria.

The proof when $A(v) = \infty$ is essentially the same.
\end{proof}

\subsection{Asymptotic Valuations}

We will often use asymptotic versions of the valuations discussed in the previous section.

\begin{defn}
Let $L$ be a divisor on $X$ with $\kappa(L) \geq 0$.  For any $v \in \Val(X)$ we define the asymptotic valuation
\begin{equation*}
v(\Vert L \Vert) := \liminf_{m \to \infty} \left\{ \left. \frac{1}{m}v(D) \right| D \in |mL| \right\}.
\end{equation*}
\end{defn}

As discussed in the introduction, we will be particularly interested in how asymptotic behavior changes as we perturb our divisor.  The following properties are verified for divisorial valuations in \cite{nakayama04} Chapter III.  Since the proof only relies on formal properties of valuations, they extend to our context with no change (see also \cite{elmnp05}, Theorem A).

\begin{thrm}[\cite{nakayama04}, Chapter III] \label{bigconefunction}
Let $X$ be a smooth projective variety and $v \in \Val^{*}(X)$.  If $L$ and $L'$ are big divisors with $L \equiv L'$ then $v(\Vert L \Vert) = v(\Vert L' \Vert)$.  Furthermore, the naturally induced function
\begin{equation*}
v(\Vert - \Vert): \mathrm{Big}(X) \to \mathbb{R}
\end{equation*}
is continuous.
\end{thrm}

\begin{cor} \label{rescaleval}
Let $X$ be a smooth projective variety and $v \in \Val^{*}(X)$.  Suppose that $L$ is a divisor with $\kappa(L) \geq 0$.  For any $t \in \mathbb{Q}_{> 0}$
\begin{equation*}
v(\Vert tL \Vert) = t v(\Vert L \Vert).
\end{equation*}
If $L$ is big, then the same statement holds for any $t \in \mathbb{R}_{>0}$.
\end{cor}

\begin{proof}
The first statement follows easily from the definition.   Since $v(\Vert - \Vert)$ is continuous on the big cone by Theorem \ref{bigconefunction}, the second statement follows from the first.
\end{proof}

The behavior of asymptotic valuations along the boundary of the pseudo-effective cone can be subtle.  Following \cite{nakayama04}, we will consider a natural extension of the function $v(\Vert - \Vert)$ to the pseudo-effective boundary.

\begin{defn}
Let $L$ be a pseudo-effective divisor on $X$.  Fix an ample divisor $A$ on $X$.  For any $v \in \Val^{*}(X)$ we define the $\sigma$-valuation
\begin{equation*}
\sigma_{v}(L) := \lim_{\epsilon \to 0} \left\{ v(\Vert L+\epsilon A \Vert) \right\}.
\end{equation*}
\end{defn}

The arguments of \cite{nakayama04} III.1.5 Lemma show that $\sigma_{v}$ is independent of the choice of $A$.  Note that $\sigma_{v}$ coincides with $v(\Vert - \Vert)$ on the big cone of $X$.

Fix a pseudo-effective divisor $L$ on $X$.  For simplicity, for a prime divisor $\Gamma$ on $X$ we write $\sigma_{\Gamma}$ instead of $\sigma_{\mathrm{ord}_{\Gamma}}$.  \cite{nakayama04} shows that as $\Gamma$ varies over all prime divisors on $X$ there are only finitely many choices of $E$ such that $\sigma_{\Gamma}(L) > 0$, allowing us to make the following definition.

\begin{defn}
Let $L$ be a pseudo-effective divisor.  Define
\begin{equation*}
N_{\sigma}(L) = \sum \sigma_{\Gamma}(L) \Gamma \qquad \qquad P_{\sigma}(L) = L - N_{\sigma}(L)
\end{equation*}
The decomposition $L = N_{\sigma}(L) + P_{\sigma}(L)$ is called the \emph{divisorial Zariski
decomposition} of $L$.
\end{defn}

\begin{defn}
Let $L$ be a pseudo-effective divisor.  We say that $L$ has a Zariski decomposition if there is a birational map $\phi: Y \to X$ such that $P_{\sigma}(\phi^{*}L)$ is nef.
\end{defn}

We will frequently use the following observation.

\begin{lem} \label{observation}
Let $L$ be a divisor with $\kappa(L) \geq 0$.  Fix an ample divisor $A$.  Then for any $\epsilon>0$ such that $A + \epsilon L$ is ample,
\begin{equation*}
v(\Vert L + A \Vert) \leq (1-\epsilon) v(\Vert L \Vert).
\end{equation*}
In particular $v(\Vert L + tA \Vert)$ is strictly decreasing in $t$.
\end{lem}

\begin{proof}
It is clear that for any ample divisor $H$ we have $v(\Vert L + H \Vert) \leq v(\Vert L \Vert)$.  Since $A+\epsilon L$ is ample, we can write
\begin{equation*}
v(\Vert L + A \Vert) = v(\Vert (1-\epsilon)L + (A+\epsilon L) \Vert) \leq (1-\epsilon) v(\Vert L \Vert).
\end{equation*}
\end{proof}

\section{Asymptotic Multiplier Ideals} \label{asymmultidealsection}

In this section we will develop the theory of asymptotic multiplier ideals for $\mathbb{R}$-divisors and demonstrate that they behave similarly to asymptotic multiplier ideals for $\mathbb{Q}$-divisors.

\subsection{Definition}
We begin by recalling the definition of a multiplier ideal of an $\mathbb{R}$-divisor.

\begin{defn}
Let $X$ be a smooth variety and let $L$ be an effective $\mathbb{R}$-divisor.  Let $\phi: Y \to X$ be a log resolution of $L$.  For $c \in \mathbb{R}_{>0}$ the multiplier ideal $\mathcal{J}(cL)$ is defined to be
\begin{equation*}
\mathcal{J}(cD) := \phi_{*}\mathcal{O}_{Y}(K_{Y/X} - \lfloor c \phi^{*}L \rfloor).
\end{equation*}
\cite{lazarsfeld04} Theorem 9.2.18 shows that the multiplier ideal is independent of the choice of log resolution.
\end{defn}

Just as for $\mathbb{Z}$-divisors, we can extend the construction to linear series.

\begin{constr}(cf.~\cite{lazarsfeld04} Definition 9.2.10)
Let $X$ be a smooth variety and let $L$ be an effective $\mathbb{R}$-divisor.  Choose a simultaneous log resolution $f: Y \to X$ of $L$ and $\mathfrak{b}(|\lfloor L \rfloor|)$.  
Write $f^{*}L \sim M + F + f^{*}\{ L \}$ where $M$ is a basepoint free divisor and $F$ is the fixed part of $| f^{*}\lfloor L \rfloor|$.  For any $c \in \mathbb{R}_{>0}$ we define
\begin{equation*}
\mathcal{J}(c|L|) = f_{*}\mathcal{O}_{Y}(K_{Y/X} - \lfloor cF + cf^{*}\{ L \} \rfloor).
\end{equation*}
Note that $\mathcal{J}(c|L|)$ is independent of the choice of resolution and only depends on the linear equivalence class of $L$.  Furthermore, if $D \sim L$ is an effective divisor then $\mathcal{J}(cD) \subset \mathcal{J}(c|L|)$.
\end{constr}

The following two lemmas can be proved in exactly the same way as for $\mathbb{Q}$-divisors.

\begin{lem}[\cite{lazarsfeld04} Proposition 9.2.26] \label{linearidealqequiv}
Let $X$ be a smooth variety, $L$ be an $\mathbb{R}$-divisor with $L \geq 0$, and $c \in \mathbb{R}_{>0}$.  Then there is some effective divisor $D$ with $D \sim_{\mathbb{Q}} L$ such that
\begin{equation*}
\mathcal{J}(c|L|) = \mathcal{J}(cD).
\end{equation*}
\end{lem}

\begin{lem}[\cite{lazarsfeld04} Lemma 11.1.1] \label{asymptoticideallemma}
Let $X$ be a smooth variety, $L$ be an $\mathbb{R}$-divisor with $L \geq 0$, and $c \in \mathbb{R}_{>0}$.  Then for any integer $k$ we have
\begin{equation*}
\mathcal{J}(c|L|) \subset \mathcal{J}\left(\frac{c}{k}|kL| \right).
\end{equation*}
\end{lem}

These two lemmas allow us to define the asymptotic multiplier ideal for $\mathbb{R}$-divisors in the same way as for $\mathbb{Q}$-divisors.

\begin{defn} \label{asymmultidealdef}
Let $X$ be a smooth variety and let $L$ be an $\mathbb{R}$-divisor with $\kappa(L) \geq 0$.  Fix $c \in \mathbb{R}_{>0}$.  Lemmas \ref{linearidealqequiv} and \ref{asymptoticideallemma} shows that the set $\{ \mathcal{J}(cD) | D \sim_{\mathbb{Q}} L, D \geq 0 \}$ is a directed set under the containment relation.  Since the underlying rings are Noetherian, there is a unique maximal element which we define to be $\mathcal{J}(c \Vert L \Vert)$.  We also will use the following convention.
\end{defn}

\begin{conv}
If $\kappa(L) = -\infty$, we set $\mathcal{J}(c \Vert L \Vert) = 0$.
\end{conv}

Note that Definition \ref{asymmultidealdef} reduces to the usual one when $L$ is a $\mathbb{Q}$-divisor.

\begin{defn}
Let $L$ be an $\mathbb{R}$-divisor with $\kappa(L) \geq 0$ and fix $c \in \mathbb{R}_{>0}$.  Definition \ref{asymmultidealdef} indicates that for some effective $D \sim_{\mathbb{Q}} L$ we have
\begin{equation*}
\mathcal{J}(c \Vert L \Vert) = \mathcal{J}(cD).
\end{equation*}
We say that $D \sim_{\mathbb{Q}} L$ computes the asymptotic multiplier ideal of $L$.
\end{defn}

One readily verifies that asymptotic multiplier ideals of $\mathbb{R}$-divisors have many of the same basic properties as in the $\mathbb{Q}$-divisor case.  We will need the following version of Nadel vanishing.

\begin{thrm} [\cite{lazarsfeld04} Remark 9.4.14]
Let $X$ be a smooth variety, let $L$ be an $\mathbb{R}$-divisor with $\kappa(L) \geq 0$, and let $D$ be an integral divisor such that $D - L$ is big and nef.  Then
\begin{equation*}
H^{i}(X,\mathcal{O}_{X}(K_{X} + D) \otimes \mathcal{J}(\Vert L \Vert)) = 0
\end{equation*}
for every $i>0$.
\end{thrm}

\begin{proof}
Let $L' \sim_{\mathbb{Q}} L$ be a divisor computing the asymptotic multiplier ideal for $L$.  Since $D - L'$ is also big and nef, we may apply the $\mathbb{R}$-divisor version of Nadel vanishing to $D$ and $L'$ to conclude.
\end{proof}

\subsection{Graded sequences of ideals}

Our next goal is to make use of the machinery of graded sequences of ideals to study asymptotic multiplier ideals for $\mathbb{R}$-divisors.  Recall that for a non-zero graded sequence of ideals $\mathfrak{a}_{\bullet}$ the multiplier ideal $\mathcal{J}(c \mathfrak{a}_{\bullet})$ is defined to be the unique maximal element of the set $\{ \mathcal{J}(\frac{c}{m} \cdot \mathfrak{a}_{m}) \}_{m \in \mathbb{Z}_{>0}}$.   Similarly, for $v \in \Val(X)$ the asymptotic valuation $v(\mathfrak{a}_{\bullet})$ is defined to be
\begin{equation*}
v(\mathfrak{a}_{\bullet}) := \liminf_{m \in \mathbb{Z}_{>0}} \frac{v(\mathfrak{a}_{m})}{m}.
\end{equation*}
Finally, given a non-zero ideal $\mathfrak{q}$, the $\mathfrak{q}$-log canonical threshold of a graded sequence of ideals $\mathfrak{a}_{\bullet}$ is defined to be
\begin{equation*}
\lct^{\mathfrak{q}}(\mathfrak{a}_{\bullet}) := \inf_{v \in \Val^{*}(X)} \frac{A(v) + v(\mathfrak{q})}{v(\mathfrak{a}_{\bullet})}.
\end{equation*}
Note that $\lct^{\mathfrak{q}}$ takes values in $\mathbb{R}_{\geq 0} \cup \{ \infty \}$, with the value $\infty$ occuring only when some $\mathfrak{a}_{m} = \mathcal{O}_{X}$.  The $\mathfrak{q}$-log canonical thresholds are closely tied to multiplier ideals: \cite{jm10} Corollary 2.16 checks that $\mathfrak{q} \subset \mathcal{J}(c \mathfrak{a}_{\bullet})$ iff $\lct^{\mathfrak{q}}(\mathfrak{a}_{\bullet}) > c$.

Given an effective $\mathbb{R}$-divisor $L$, define
\begin{equation*}
\mathfrak{a}_{m} = \mathfrak{b}(|\lfloor mL \rfloor|) \cdot \mathcal{O}_{X}(-\lceil \{ mL \} \rceil).
\end{equation*}

\begin{prop}
The ideals $\{ \mathfrak{a}_{\bullet} \}$ form a graded sequence with
\begin{enumerate}
\item $v(\mathfrak{a}_{\bullet}) = v(\Vert L \Vert)$ for any valuation $v \in \Val(X)$ and
\item $\mathcal{J}( \mathfrak{a}_{\bullet}) = \mathcal{J}(\Vert L \Vert)$.
\end{enumerate}
\end{prop}

\begin{proof}
We first must check that the $\mathfrak{a}_{\bullet}$ form a graded sequence.  Fix integers $i,j>0$.  Note that there is a natural injection
\begin{equation*}
\mathfrak{b}(|\lfloor iL \rfloor|)  \mathfrak{b}(| \lfloor jL \rfloor |) \cdot \mathcal{O}_{X}(-\lfloor \{iL\} + \{jL\} \rfloor) \to \mathfrak{b}(| \lfloor (i+j)L \rfloor |)
\end{equation*}
and that
\begin{align*}
\lceil \{iL\}  \rceil + \lceil \{jL\} \rceil - \lfloor \{ iL \} + \{ jL \} \rfloor  & \geq  \lceil \{iL\} + \{jL\} \rceil - \lfloor \{ iL \} + \{ jL \} \rfloor \\
& = \lceil \{ (i+j) L \} \rceil.
\end{align*}
Combining the two statements yields the desired inclusion $\mathfrak{a}_{i}  \mathfrak{a}_{j} \subset \mathfrak{a}_{i+j}$.

We next prove (1).  Set $D = \lceil L \rceil - \lfloor L \rfloor$.  Since $v(\mathfrak{a}_{m}) = v(|\lfloor mL \rfloor|) + v(\lceil \{ mL \} \rceil)$ we have
\begin{equation*}
v(|mL|) \leq v(\mathfrak{a}_{m}) \leq v(|mL|) + v(D).
\end{equation*}
Dividing by $m$ and taking the liminf demonstrates the equality.

Finally, we prove (2).  As $\lceil \{ mL \} \rceil \geq \{ mL \}$, it is clear that $\mathcal{J}(c |mL|) \supset \mathcal{J}(c \cdot \mathfrak{a}_{m})$, and thus $\mathcal{J}(\Vert L \Vert) \supset \mathcal{J}(\mathfrak{a}_{\bullet})$.  Conversely, as before define $D = \lceil L \rceil - \lfloor L \rfloor$.  Suppose that $L' \in \frac{1}{p}|pL|$ computes the asymptotic multiplier ideal of $L$.  Choose $\epsilon$ sufficiently small so that $\mathcal{J}(L' + \epsilon D) = \mathcal{J}(L')$.  Then for $m > \frac{1}{\epsilon}$ we have
\begin{align*}
\mathcal{J}\left(\frac{1}{pm} \mathfrak{a}_{pm} \right) & \supset \mathcal{J}\left(L' + \frac{1}{pm} D\right) \\
& = \mathcal{J}(\Vert L \Vert).
\end{align*}
\end{proof}

We may now apply the results of \cite{jm10} to $\mathbb{R}$-divisors.  The following is an immediate consequence of \cite{jm10} Corollary 6.4.

\begin{prop} \label{gradedcontinuity}
Let $L$ be an $\mathbb{R}$-divisor on $X$ with $\kappa(L) \geq 0$.  Then the function $v \mapsto v(\Vert L \Vert)$ is continuous on $\{ v \in \Val_{X} | A(v) < \infty \}$.
\end{prop}

Another consequence of \cite{jm10} is that the log canonical threshold of the asymptotic multiplier ideal of an $\mathbb{R}$-divisor can be computed using asymptotic valuations.

\begin{thrm}[\cite{jm10}, Theorem 7.3] \label{qvaluativedescription}
Let $L$ be a pseudo-effective $\mathbb{R}$-divisor on $X$ with $\kappa(L) \geq 0$.  Fix an ideal sheaf $\mathfrak{q}$ on $X$.  We have $\mathfrak{q} \subset \mathcal{J}(\Vert L \Vert)$ iff
\begin{equation*}
v(\mathfrak{q}) > v(\Vert L \Vert) - A(v)
\end{equation*}
for every $v \in \Val^{*}(X)$.
\end{thrm}

In particular, a regular function $f$ on a Zariski open subset $U \subset X$ lies in $\mathcal(\Vert L \Vert)$ iff $v(f) > v(\Vert L \Vert) - A(v)$ for every valuation $v \in \Val^{*}(X)$.

\begin{proof}
We may assume $L$ is effective and define $\mathfrak{a}_{\bullet}$ as above.  \cite{jm10} Theorem 7.3 shows that $\lct^{\mathfrak{q}}(\mathfrak{a}_{\bullet})$ is computed by some valuation in $\Val^{*}(X)$.  Thus the valuation inequality is equivalent to $\lct^{\mathfrak{q}}(\mathfrak{a}_{\bullet}) > 1$, i.e.~$\mathfrak{q} \subset \mathcal{J}(\Vert L \Vert)$.
\end{proof}

Using Theorem \ref{qvaluativedescription}, it is easy to check that the following well-known lemmas for $\mathbb{Q}$-divisors also hold for $\mathbb{R}$-divisors.

\begin{cor}
Let $L$ be a big $\mathbb{R}$-divisor and suppose that $L' \equiv L$.  Then for any $c>0$
\begin{equation*}
\mathcal{J}(c\Vert L \Vert) = \mathcal{J}(c \Vert L' \Vert).
\end{equation*}
\end{cor}

\begin{proof}
Combine Theorem \ref{qvaluativedescription} and Theorem \ref{bigconefunction}.
\end{proof}

\begin{cor} \label{plusamplecontainment}
Let $L$ be a pseudo-effective $\mathbb{R}$-divisor and let $A$ be an ample $\mathbb{R}$-divisor.  Then for any $c>0$
\begin{equation*}
\mathcal{J}(c \Vert L \Vert) \subset \mathcal{J}(c \Vert L + A \Vert).
\end{equation*}
\end{cor}

\begin{proof}
Combine Theorem \ref{qvaluativedescription} and Lemma \ref{observation}.
\end{proof}

\begin{cor} \label{nestedideals}
Let $L$ be a pseudo-effective $\mathbb{R}$-divisor and let $c>0$.  Then for any $\epsilon \in \mathbb{Q}_{> 0}$ we have
\begin{equation*}
\mathcal{J}(c\Vert L \Vert) \supset \mathcal{J}(c\Vert (1+\epsilon)L \Vert)
\end{equation*}
and the two ideals coincide for sufficiently small $\epsilon$.  
If $L$ is big, the same statement holds for any $\epsilon \in \mathbb{R}_{>0}$.
\end{cor}

\begin{proof}
We may assume that $\kappa(L) \geq 0$.  The containment $\supset$ follows from Theorem \ref{qvaluativedescription} and Corollary \ref{rescaleval}. Furthermore, suppose that $D \sim_{\mathbb{Q}} L$ is an effective divisor computing the asymptotic multiplier ideal $\mathcal{J}(c \Vert L \Vert)$.  Choose $\epsilon \in \mathbb{Q}_{>0}$ sufficiently small so that $\mathcal{J}(cD) = \mathcal{J}(c(1+\epsilon)D)$ (see \cite{lazarsfeld04} Example 9.2.30).  Then Definition \ref{asymmultidealdef} indicates that $\mathcal{J}(c(1+\epsilon)D) \subset \mathcal{J}(c \Vert (1+\epsilon)L \Vert)$, giving the reverse containment.

When $L$ is big, Corollary \ref{rescaleval} implies that the same argument will work for any $\epsilon \in \mathbb{R}_{>0}$.
\end{proof}

\section{Perturbations of Asymptotic Multiplier Ideals} \label{perturbedsection}

We now turn to the question of how the asymptotic multiplier ideal $\mathcal{J}(\Vert L \Vert)$ behaves as we vary the divisor $L$.  Our goal is to show that $\overline{NE}^{1}(X)$ admits a chamber decomposition such that $\mathcal{J}(\Vert - \Vert)$ is constant on the interior of each region.

\subsection{Perturbed Ideals}

This section explores how multiplier ideals behave upon small perturbations.  We first subtract a small ample divisor from $L$.

\begin{prop} \label{augmentedideal}
Let $X$ be a smooth variety and let $L$ be a big divisor.  Fix an ample divisor $A$.  Then there is some $\epsilon > 0$ such that
\begin{equation*}
\mathcal{J}(\Vert L - \epsilon A \Vert) = \mathcal{J}(\Vert L \Vert).
\end{equation*}
\end{prop}

\begin{proof}
For any $\epsilon > 0$ the containment $\subset$ follows from Corollary \ref{plusamplecontainment}.  Conversely, fix some $m>0$ such that $mL - A$ is big.  Suppose that $\epsilon$ is a sufficiently small positive number so that $m\epsilon < 1$.  Then
\begin{align*}
v(\Vert L - \epsilon A \Vert) &= v(\Vert (1- m\epsilon)L + \epsilon (mL - A) \Vert) \\
& = (1-m\epsilon) v\left(\left\Vert L + \frac{\epsilon}{1-m\epsilon}(mL-A) \right\Vert \right)
\end{align*}
for every valuation $v \in \Val^{*}(X)$.  Note that Corollary \ref{nestedideals} shows that we may choose $\epsilon$ sufficiently small so that $\mathcal{J}(\Vert L + \frac{\epsilon}{1-m\epsilon} (mL-A) \Vert) = \mathcal{J}(\Vert L \Vert)$.  By comparing valuations Theorem \ref{qvaluativedescription} yields
\begin{equation*}
\mathcal{J}(\Vert L - \epsilon A \Vert) \supset \mathcal{J}\left( \left\Vert L + \frac{\epsilon}{1-m\epsilon} (mL-A) \right\Vert \right) = \mathcal{J}(\Vert L \Vert)
\end{equation*}
\end{proof}

We next add a small ample divisor.  In order to understand this case, we will need a slight generalization of a construction from \cite{hacon04} Proposition 5.1.

\begin{thrm} \label{haconconstr}
Let $\mathcal{H} \subset N^{1}(X)$ be a bounded set of classes of pseudo-effective divisors.  Then the set of ideals $\{ \mathcal{J}(\Vert L \Vert) \}_{c_{1}(L) \in \mathcal{H}}$ does not have any infinite descending sequence.
\end{thrm}

\begin{proof}
By removing those divisors with $\mathcal{J}(\Vert L \Vert) = 0$, we may assume that every divisor $L$ with class in $\mathcal{H}$ has $\kappa(L) \geq 0$ and is non-zero.

Since $\mathcal{H}$ is bounded, we may choose an ample $\mathbb{Z}$-divisor $D$ such that for any $L$ with class in $\mathcal{H}$ we have that $D - L$ is ample.  Thus for any $L$ with class in $\mathcal{H}$ there is an ample divisor $A'$ such that
\begin{equation*}
K_{X} + D + nH \equiv L + A' + K_{X} + nH.
\end{equation*}
Nadel vanishing implies that
\begin{equation*}
H^{i}\left(X,\mathcal{O}_{X}(K_{X} + D + nH) \otimes \mathcal{J}\left(\left\Vert L \right\Vert\right)\right) = 0
\end{equation*}
for all $i>0$.  By Castelnuovo-Mumford regularity, the sheaves $\mathcal{O}_{X}(K_{X} + D + nH) \otimes \mathcal{J}(\Vert L \Vert)$ are globally generated for any $L$ with class in $\mathcal{H}$.  In particular, the subsheaves
\begin{equation*}
\mathcal{O}_{X}(K_{X} + D + nH) \otimes \mathcal{J}(\Vert L \Vert) \subset \mathcal{O}_{X}(K_{X} + D + nH)
\end{equation*}
are uniquely determined by their global sections
\begin{equation*}
H^{0}\left(X,\mathcal{O}_{X}(K_{X} + D + nH) \otimes \mathcal{J}\left( \left\Vert L \right\Vert\right)\right)
\subset H^{0}\left(X,\mathcal{O}_{X}(K_{X} + D + nH) \right).
\end{equation*}
Since $H^{0}(X,\mathcal{O}_{X}(K_{X} + D + nH))$ is finite dimensional, there is no infinite descending chain of subspaces, hence also no infinite descending chain of multiplier ideals.
\end{proof}

This construction allows us to define the multiplier ideal corresponding to a small perturbation of $L$ by an ample divisor.

\begin{defn}[\cite{hacon04}] \label{perturbedidealdefn}
Let $X$ be a smooth variety and let $L$ be a pseudo-effective divisor.  Fix any ample divisor $A$.  We define the perturbed ideal $\mathcal{J}_{-}(L)$ to be the smallest ideal in the finite descending chain of ideal sheaves $\{ \mathcal{J}(\Vert L + \frac{1}{t}A \Vert) \}_{t=1}^{\infty}$.  Note that this ideal is independent of the choice of $A$.
\end{defn}

\begin{rmk}
There is no reason to expect $\mathcal{J}_{-}(L)$ to control vanishing theorems for $L$ unless we include a perturbation by an ample (or big) divisor.  We thus avoid calling $\mathcal{J}_{-}(L)$ a multiplier ideal, reserving this notation for ideals that can be constructed by metrics.
\end{rmk}

It is clear that $\mathcal{J}(\Vert L \Vert) \subset \mathcal{J}_{-}(L)$ and that $\mathcal{J}_{-}(L)$ only depends on the numerical class of $L$.  

\begin{exmple} \label{perturbedexample}
It may happen that $\mathcal{J}(\Vert L \Vert) \subsetneq \mathcal{J}_{-}(L)$ even for a big divisor $L$.  For example, suppose that $\pi: S \to \mathbb{P}^{2}$ is the blow-up of a point.  Let $E$ denote the exceptional divisor and let $H$ denote the hyperplane class on $\mathbb{P}^{2}$.  Then
\begin{equation*}
\mathcal{J}(\Vert \pi^{*}H + E \Vert) = \mathcal{O}_{S}(-E) \subsetneq \mathcal{O}_{S} = \mathcal{J}_{-}(\pi^{*}H + E).
\end{equation*}
\end{exmple}


\begin{lem} \label{perturbedgivesdiminishedlocus}
Let $L$ be a pseudo-effective divisor.  We have an equality of sets
\begin{equation*}
\bigcup_{m \in \mathbb{Z}_{> 0}} V(\mathcal{J}_{-}(mL)) = \mathbf{B}_{-}(L).
\end{equation*}
\end{lem}

\begin{proof}
\cite{elmnp05} Corollary 2.10 shows that for a big $\mathbb{Q}$-divisor $D$ we have
\begin{equation*}
\bigcup_{m \in \mathbb{Z}_{> 0}} V(\mathcal{J}(\Vert mD \Vert)) = \mathbf{B}_{-}(D).
\end{equation*}
The proof extends easily to the case of an arbitrary big $\mathbb{R}$-divisor $D$.
Then for a fixed ample divisor $A$,
\begin{align*}
\mathbf{B}_{-}(L) & = \bigcup_{\epsilon > 0} \mathbf{B}_{-}(L + \epsilon A) \\
& = \bigcup_{\epsilon > 0} \bigcup_{m  \in \mathbb{Z}_{> 0}} V(\mathcal{J}(\Vert m(L+\epsilon A) \Vert)) \\
& = \bigcup_{m  \in \mathbb{Z}_{> 0}} V(\mathcal{J}_{-}(mL))
\end{align*}
\end{proof}

The perturbed ideal has a natural interpretation in terms of valuations.

\begin{prop} \label{perturbedvaluation}
Let $L$ be a pseudo-effective $\mathbb{Q}$-divisor on $X$.  Fix an ideal sheaf $\mathfrak{q}$ on $X$.  Then $\mathfrak{q} \subset \mathcal{J}_{-}(L)$ iff
\begin{equation*} \label{perturbed} \tag{*}
v(\mathfrak{q}) \geq \sigma_{v}(L) - A(v)
\end{equation*}
for every $v \in \Val^{*}(X)$.  In fact, for the reverse implication it suffices to check \eqref{perturbed} for every divisorial valuation $v \in \Val^{*}(X)$.
\end{prop}

\begin{proof}
Suppose first that $\mathfrak{q} \subset \mathcal{J}_{-}(L)$.  For every $\epsilon > 0$ we have $v(\mathfrak{q}) > v(\Vert L + \epsilon A \Vert) - A(v)$.  Taking limits proves one implication.  Conversely, if $v(\mathfrak{q}) \geq \sigma_{v}(L) - A(v)$ for every $v \in \Val^{*}(X)$, then for any $\epsilon > 0$ we have $\mathfrak{q} \subset \mathcal{J}(\Vert L + \epsilon A \Vert)$ by Theorem \ref{qvaluativedescription} and Lemma \ref{observation}.

Finally, suppose that $w(\mathfrak{q}) \geq \sigma_{w}(L) - A(w)$ for every divisorial valuation $w \in \Val^{*}(X)$.  For an arbitrary $v \in \Val^{*}(X)$, Lemma \ref{divdensity} states that there is a net of divisorial valuations $w_{\alpha}$ converging to $v$ such that $A(v) = \lim_{\to} A(w_{\alpha})$.  Using the continuity of $v(\mathfrak{q})$ and the lower semi-continuity of $\sigma_{v}(L)$ as functions of $v$, we obtain
\begin{equation*}
v(\mathfrak{q}) \geq \sigma_{v}(L) - A(v)
\end{equation*}
by taking limits over this net.  We then conclude by the previous paragraph.
\end{proof}

Since $\sigma_{v}(L)$ is a birational invariant, we obtain a birational transformation rule for $\mathcal{J}_{-}(L)$.

\begin{cor} \label{perturbedbirational}
Let $L$ be a pseudo-effective divisor and let $\phi: Y \to X$ be a birational map from a smooth variety $Y$.  Then
\begin{equation*}
\mathcal{J}_{-}(L) = \phi_{*}(\mathcal{O}_{Y}(K_{Y/X}) \otimes \mathcal{J}_{-}(\phi^{*}L)).
\end{equation*}
\end{cor}

\subsection{Balanced Decomposition}

We next show that the big cone admits a chamber decomposition such that $\mathcal{J}(\Vert L \Vert)$ is constant for $L$ in the interior of a chamber.  We start with a local description.

\begin{prop} \label{localdecomp}
Let $L$ be a big divisor.  There is a neighborhood $U \subset N^{1}(X)$ of $c_{1}(L)$ such that
for any divisor $D$ with $c_{1}(D) \in U$ we have
\begin{equation*}
\mathcal{J}(\Vert L \Vert) \subset \mathcal{J}(\Vert D \Vert) \subset \mathcal{J}_{-}(L).
\end{equation*}
\end{prop}


\begin{proof}
Apply Proposition  \ref{augmentedideal} to find some small ample divisor $H$ so that $\mathcal{J}(\Vert L - H \Vert) = \mathcal{J}(\Vert L \Vert)$.  There is an open neighborhood $U$ of $c_{1}(L)$ such that $\overline{U} \subset c_{1}(L-H)+Amp(X)$.   
Note that every $D$ with $c_{1}(D) \in U$ is of the form $D \equiv L-H+A$ for some ample $A$.  Thus $\mathcal{J}(\Vert L \Vert) \subset \mathcal{J}(\Vert D \Vert)$ by Corollary \ref{plusamplecontainment}.

Similarly, if we choose an ample $A$ so that $\mathcal{J}(\Vert L + A \Vert) = \mathcal{J}_{-}(L)$, then by shrinking $U$ we may ensure that $U \subset c_{1}(L + A) - Amp(X)$ yielding the second containment.
\end{proof}

To translate Proposition \ref{localdecomp} to the global setting, we introduce the following definition.

\begin{defn}
Let $L$ be a big divisor.  We say that $L$ is balanced if $\mathcal{J}(\Vert L \Vert) = \mathcal{J}_{-}(L)$.
\end{defn}

\begin{thrm} 
There is a unique collection of closed connected subsets $C_{i} \subset \overline{NE}^{1}(X)$ such that
\begin{enumerate}
\item $C_{i} = \overline{C_{i}^{\circ}}$ and $\overline{NE}^{1}(X) = \bigcup C_{i}$.
\item A big divisor $L$ is balanced precisely when $c_{1}(L)$ is in the interior of one of the $C_{i}$.
\end{enumerate}
We call the $C_{i}$ balanced regions.
\end{thrm}

\begin{proof}
We first show that the set of classes of balanced divisors is open and dense in $\overline{NE}^{1}(X)$.  The openness follows from Proposition \ref{localdecomp} (1).  To show density, suppose that $L$ is any pseudo-effective divisor.  Fix an ample divisor $A$ and choose $\epsilon$ sufficiently small so that $\mathcal{J}_{-}(L) = \mathcal{J}(\Vert L + \epsilon A \Vert)$.  Then by Proposition \ref{localdecomp} the divisor $L+\tau A$ is balanced for any positive $\tau < \epsilon$.

Define $C_{i}$ to be the closure of a connected component of the set of balanced divisors.  Since the set of classes of balanced divisors is open and dense, the statement follows.  
\end{proof}

While the chamber decomposition behaves well inside the big cone, its behavior along the pseudo-effective boundary is more subtle.  We return to this question in Section \ref{diminishedsection}.

\begin{rmk}
In \cite{elmnp05} a divisor is called stable if $\mathbf{B}_{-}(L) = \mathbf{B}_{+}(L)$.  The authors show that $\overline{NE}^{1}(X)$ can be decomposed into regions in which every divisor is stable.    Despite the close relationship between $\mathbf{B}_{-}(L)$ and asymptotic multiplier ideals, the two decompositions have somewhat different behavior.

For example, a big and nef divisor is always balanced (and thus contained in the \emph{interior} of a balanced region).  However, a big and nef divisor that is not ample will be on the boundary of a stable region.
\end{rmk}



As a consequence of Theorem \ref{localdecomp} we show that $\mathfrak{q}$-log canonical thresholds are continuous functions on the big cone.

\begin{thrm} \label{lctcont}
Let $P \subset \overline{NE}^{1}(X)$ denote the open set of classes that are big but not nef.  Then $\lct^{\mathfrak{q}}(\Vert - \Vert)$ is a continuous function on $P$.
\end{thrm}

The theorem does not apply to big and nef divisors $A$ since these are precisely the divisors with $\lct^{\mathfrak{q}}(\Vert A \Vert) = \infty$.  This distinction is somewhat artificial: the Arnold multiplicity (i.e.~the inverse of the log canonical threshold) is a continuous function everywhere on the big cone.

\begin{proof}
Recall that $\lct^{\mathfrak{q}}(\Vert - \Vert)$ is defined as
\begin{equation*}
\lct^{\mathfrak{q}}(\Vert L \Vert) := \inf_{v \in \Val^{*}(X)} \frac{A(v) + v(\mathfrak{q})}{v(\Vert L \Vert)}.
\end{equation*}
Since $v(\Vert - \Vert)$ is continuous on the big cone, $\lct^{\mathfrak{q}}(\Vert - \Vert)$ is upper semi-continuous on the big cone and it suffices to prove lower semi-continuity.

Let $L$ be a big divisor and $c$ be a positive constant such that $\mathfrak{q} \subset \mathcal{J}(c\Vert L \Vert)$.  By Theorem \ref{localdecomp} there is a neighborhood $U$ of $[L]$ such that $\mathfrak{q} \subset \mathcal{J}(c \Vert D \Vert)$ for every $D$ with class in $U$.  Equivalently, if $\lct^{\mathfrak{q}}(\Vert L \Vert) > c$ then $\lct^{\mathfrak{q}}(\Vert D \Vert) > c$ for every $D$ with class in $U$, showing lower semi-continuity on $P$.
\end{proof}

\section{Analytic Multiplier Ideals} \label{analyticsection}

In this section we recall how valuation theory can be used to relate algebraic and analytic multiplier ideals.  Instead of metrics we will use positive currents which are more suitable for working with $\mathbb{R}$-divisors.  We then apply the results of \cite{boucksom04} to relate multiplier ideals of currents of minimal singularities to divisorial valuations.

Most of the material in this section is not new; we include it as a reminder about the basic facts concerning analytic multiplier ideals and to identify precisely the results we will need later.  A thorough introduction to this circle of ideas can be found in \cite{demailly12}.

\subsection{Positive Currents}

A $(p,q)$-current is a differential $(p,q)$-form with distribution coefficients.  Just as with smooth forms, any $(p,q)$-current can be written locally as $T = \overline{\partial} u$ for some $(p,q-1)$-current $u$.   Thus, one can define cohomology spaces by quotienting the space of $\overline{\partial}$-closed $(p,q)$-currents by the $\overline{\partial}$-exact ones.  These cohomology spaces turn out to be isomorphic to the usual Dolbeault cohomology $H^{p,q}(X,\mathbb{C})$.

In the special case of $(p,p)$-currents, we say that $T$ is real if it is fixed by the natural involution induced by the complex structure.  Since we will only be interested in $\overline{\partial}$-closed real currents $T$, from now on we will simply use ``current'' to denote such an object.  We can naturally identify the cohomology class of a (closed) $(p,p)$-current inside of $H^{p,p}(X,\mathbb{R}) \subset H^{2p}(X,\mathbb{R})$.  In particular, to any $(1,1)$-current $T$ we can associate a class $c_{1}(T)$ in $N^{1}(X)$.

\begin{defn}
Let $X$ be an $n$-dimensional complex projective manifold and let $T$ be a $(1,1)$-current.  We say that $T$ is positive and write $T \geq 0$ if it has non-negative pairing against every form
\begin{equation*}
\theta = i \alpha_{1} \wedge \overline{\alpha_{1}} \wedge \ldots \wedge i\alpha_{n-1} \wedge \overline{\alpha_{n-1}}.
\end{equation*}
\end{defn}

Positive $(1,1)$-currents are precisely those that can be written locally as $\frac{i}{\pi} \partial \overline{\partial} \varphi$ for a plurisubharmonic function $\varphi$ known as a local weight.  The following theorem of Demailly indicates that positive currents are closely linked to pseudo-effective divisors.

\begin{thrm} [\cite{demailly92}, Proposition 4.2]
Let $X$ be a smooth variety and let $L$ be an $\mathbb{R}$-divisor on $X$.  Then $L$ is pseudo-effective iff there is a positive current $T$ such that $c_{1}(T) = c_{1}(L)$.
\end{thrm}

As suggested by the previous theorem, to a pseudo-effective $\mathbb{R}$-divisor $L$ it is natural to associate the collection of positive $(1,1)$-currents
\begin{equation*}
\mathcal{T}(L) := \{ T \geq 0 \, | \, c_{1}(T) = c_{1}(L) \}.
\end{equation*}

When $L$ is a Cartier divisor, this association is justified by the following lemma.

\begin{lem}[\cite{eckl04}, Lemma 4.1] \label{numericallytrivialmultiplierideal}
Let $X$ be a smooth variety and $L$ a line bundle on $X$.  For every closed positive current $T$ with $c_{1}(T)=c_{1}(L)$ there is a singular hermitian metric $h$ on $L$ with curvature current $T$.
\end{lem}

\subsection{Multiplier Ideals}

\begin{defn}
Let $X$ be a smooth variety and let $T$ be a positive $(1,1)$-current on $X$.  The multiplier ideal $\mathcal{J}(T)$ is the unique ideal satisfying the the local conditions at every point $x$
\begin{equation*}
\mathcal{J}(\phi)_{x} = \{ f \in \mathcal{O}_{X,x} \; | \; \exists U \ni x \textrm{ such that } \int_{U} |f|^{2}e^{-2\varphi} dV < \infty  \}
\end{equation*}
where $\varphi$ is a local weight function.
\end{defn}

The following well-known transformation rule controls how multiplier ideals behave under birational maps.

\begin{lem} \label{birationaltransformationformula}
Let $\phi: Y \to X$ be a birational map of smooth varieties and let $T$ be a positive current on $X$.  Then
\begin{equation*}
\mathcal{J}(T) = \phi_{*} \left( \mathcal{O}_{Y}(K_{Y/X}) \otimes \mathcal{J}(\phi^{*}T) \right).
\end{equation*}
\end{lem}

Following \cite{boucksom04}, for a prime divisor $D$ on $X$ we define
\begin{equation*}
\nu(T,D) := \min_{x \in D} \nu(T,x)
\end{equation*}
where $\nu(T,x)$ denotes the Lelong number of $T$ at $x$.  Using the Siu decomposition, we see this minimum is achieved by a very general point $x \in D$.

\begin{lem} \label{codimensiononecalculation}
Let $X$ be a smooth variety and let $T$ be a current on $X$.  Suppose that $\nu(T,D) > 0$ for some divisor $D$ on $X$.  Then $v_{D}(\mathcal{J}(T)) = \lfloor \nu(T,D) \rfloor$.
\end{lem}

\begin{proof}
Denote the Siu decomposition of $T$ as
\begin{equation*}
T = R + \sum_{i=1}^{\infty} \nu(T,D_{i})[D_{i}]
\end{equation*}
where $[D_{i}]$ denotes the current of integration against the prime divisor $D_{i}$ and $E_{c}(R)$ has codimension at least $2$.  A very general point $x \in D$ avoids every other $D_{i}$ and $E_{c}(R)$.  Locally near $x$ membership in $\mathcal{J}(T)$ is determined by integrability conditions against a weight function for $\nu(T,D)[D]$.  Thus we are reduced to the familiar one-dimensional calculation.
\end{proof}

There are deeper results relating valuations to analytic multiplier ideals in \cite{fj04}, \cite{fj05}, and in \cite{bfj08}, but we do not need them.

\subsection{Currents of Minimal Singularities}

Among all the positive currents numerically equivalent to a divisor $L$, there are some with the ``best singularities'' in the following sense.

Suppose that $T_{1}$ and $T_{2}$ are two positive $(1,1)$-currents defined locally in suitable charts by weight functions $\phi_{1}$ and $\phi_{2}$.  We say that $T_{1}$ is less singular than $T_{2}$, and write $T_{1} \preceq T_{2}$, if there is some constant $C$ such that $\phi_{2} \leq \phi_{1} + C$.  Note that the ambiguity in the choice of weight function is accounted for by the presence of $C$ so that this definition only depends on the two currents.

Fix a divisor $L$ and consider the set $\mathcal{T}$ of positive currents $T$ with $c_{1}(T) = c_{1}(L)$.  Over all currents in $\mathcal{T}$, there is a minimal current $T_{min}$ under the relation $\preceq$; see for example \cite{dps01} Theorem 1.5.  We call $T_{min}$ a current of minimal singularities in the numerical class of $L$.  A metric of minimal singularities  has the smallest possible Lelong numbers, and the largest possible multiplier ideal, among all positive currents $T$ with the same numerical class.

The following proposition is an immediate consequence of \cite{boucksom04} Theorem 5.4.

\begin{prop}[\cite{boucksom04}, Theorem 5.4] \label{lelongbigequality}
Suppose that $L$ is a big divisor.  Then
\begin{equation*}
\nu(T_{min},D) = v_{D}(\Vert L \Vert)
\end{equation*}
for any prime divisor $D$ on $X$.
\end{prop}

\begin{cor} \label{lelongcor}
Let $X$ be smooth, $L$ a pseudo-effective divisor.  Then for any birational map $\phi: Y \to X$ and any prime divisor $D$ on $Y$ we have
\begin{equation*}
\nu(\phi^{*}T_{min},D) \geq \sigma_{D}(L).
\end{equation*}
\end{cor}

\begin{proof}
Note that $\nu(\phi^{*}T_{min,L},D) \geq \nu(T_{min,\phi^{*}L},D)$.  Thus it suffices to consider the case when $Y = X$.

Let $\omega$ be a K\"ahler form on $X$ in the numerical class of an ample line bundle $A$.  Then $T_{min,L} + \epsilon \omega \succeq T_{min,L + \epsilon A}$.  By upper semicontinuity of Lelong numbers (see \cite{demailly93}, Proposition 3.12):
\begin{equation*}
\nu(\lim_{\epsilon \to 0} T_{min,L} + \epsilon \omega,D) \geq \lim_{\epsilon \to 0} \nu(T_{min,L} + \epsilon \omega,D) \geq \lim_{\epsilon \to 0} \nu(T_{min,L + \epsilon A},D).
\end{equation*}
By Proposition \ref{lelongbigequality} $\nu(T_{min,L},D) \geq \sigma_{D}(L)$.
\end{proof}

\begin{proof}[Proof of Theorem \ref{idealcomparison}:]
Lemma \ref{codimensiononecalculation} and Corollary \ref{lelongcor} show that for every birational map $\phi: Y \to X$ and every prime divisor $D$ on $Y$ we have
\begin{equation*}
v_{D}(\mathcal{J}(T_{min})) \geq \lfloor \sigma_{D}(L) \rfloor.
\end{equation*}
Since $A(v_{D})$ is at least $1$, we obtain
\begin{equation*}
v_{D}(\mathcal{J}(T_{min})) >  \sigma_{D}(L)  - A(v_{D}),
\end{equation*}
and thus similarly for any divisorial valuation.  By Proposition \ref{perturbedvaluation} we have that $\mathcal{J}(T_{min}) \subset \mathcal{J}_{-}(L)$.
\end{proof}

For divisors on the pseudo-effective boundary, the containments
\begin{equation*}
\mathcal{J}(\Vert L \Vert) \subset \mathcal{J}(h_{min}) \subset \mathcal{J}_{-}(L)
\end{equation*}
may be strict (and similarly for the conjectural containment $\mathcal{J}(h_{min}) \subset \mathcal{J}_{\sigma}(L)$).

\begin{exmple}[\cite{del00}, pgs. 8-9]
It may happen that $\mathcal{J}(\Vert L \Vert)$ is properly contained in $\mathcal{J}(h_{min})$.
For example, let $C$ be an elliptic curve and let $\tau$ be a numerically trivial non-torsion line bundle.  Consider the projective bundle $X = \mathbb{P}_{C}(\mathcal{O} \oplus \tau)$ and let $L$ be the zero section.  For every $m \geq 1$, $\mathcal{O}_{Y}(mL)$ has the unique section $mL$ so $\mathcal{J}(\Vert L \Vert) = \mathcal{O}_{X}(-L)$.  However, since $\mathcal{O} \oplus \tau$ is a unitary flat bundle, $\mathcal{O}_{X}(L)$ has a smooth semipositive metric and thus $\mathcal{J}(h_{min}) = \mathcal{O}_{X}$.
\end{exmple}

\begin{exmple}[\cite{dps01}, Example 2.5] \label{dps01exmple}
It is also possible that $\mathcal{J}(h_{min})$ is properly contained in $\mathcal{J}_{-}(L)$ (and the diminished ideal $\mathcal{J}_{\sigma}(L)$).  For example,
let $C$ be an elliptic curve and let $X \to C$ be the projective
bundle corresponding to the unique nonsplit extension
\begin{equation*}
0 \to \mathcal{O}_{C} \to E \to \mathcal{O}_{C} \to 0
\end{equation*}
Let $L$ be the zero section.  Combining the multiplier ideal version of Hard Lefschetz with an explicit cohomology calculation, \cite{dps01} proves that $L$ can not carry a metric with $\mathcal{J}(h_{min}) = \mathcal{O}_{X}$.  However, since $L$ is nef we have $\mathcal{J}_{-}(L) = \mathcal{O}_{X}$ (and also $\mathcal{J}_{\sigma}(L) = \mathcal{O}_{X}$).
\end{exmple}

\section{Diminished Ideals} \label{diminishedsection}

One disadvantage of $\mathcal{J}_{-}(L)$ is that the perturbation by an ample divisor increases the ``positivity'' of our divisor.  As shown by Example \ref{perturbedexample}, this increase is unavoidable even when $L$ is big.  In this section we correct this deficiency by working with the upper-regularization of $\mathcal{J}_{-}(L)$.  We call this ideal the diminished ideal and denote it by $\mathcal{J}_{\sigma}(L)$.

\begin{lem}
Let $L$ be a pseudo-effective divisor.  The ideals $\mathcal{J}_{-}((1 + \epsilon)L)$ form an ascending chain as $\epsilon > 0$ decreases.
\end{lem}

\begin{proof}
Fix an ample divisor $A$ and $0<\epsilon_{0}<\epsilon_{1}$.  One simply notes that for some sufficiently small $t$ one has
\begin{equation*}
\mathcal{J}_{-}((1+\epsilon_{0})L) = \mathcal{J}(\Vert (1+\epsilon_{0})L + (1+\epsilon_{0})tA \Vert)
\end{equation*}
and
\begin{equation*}
\mathcal{J}_{-}((1+\epsilon_{1})L) = \mathcal{J}(\Vert (1+\epsilon_{1})L + (1+\epsilon_{1})tA \Vert).
\end{equation*}
Lemma \ref{nestedideals} gives the desired conclusion.
\end{proof}

\begin{defn} \label{sigmadefn}
Let $L$ be a pseudo-effective divisor.  The diminished ideal $\mathcal{J}_{\sigma}(L)$ is defined to be the maximal element of the set $\{ \mathcal{J}_{-}((1+\epsilon)L) \}_{\epsilon > 0}$.
\end{defn}

Note that $\mathcal{J}_{\sigma}(L)$ only depends on the numerical class of $L$.  Furthermore, it is clear from the construction that $\mathcal{J}_{\sigma}(L) = \mathcal{J}_{\sigma}((1+\epsilon)L)$ for sufficiently small positive $\epsilon$.

We will soon prove that when $L$ is big $\mathcal{J}_{\sigma}(L) = \mathcal{J}(\Vert L \Vert)$.  Thus the diminished ideal can be thought of as a ``continuous'' extension of the asymptotic multiplier ideal from the big cone to the pseudo-effective boundary.  It turns out that the diminished ideal shares many desirable properties with asymptotic multiplier ideals of big divisors; see \cite{lazarsfeld04} Remark 11.1.10.

\begin{exmple}
Suppose that $X$ has dimension $2$.  Then for any big divisor $L$ we have $\mathcal{J}(\Vert L \Vert) = \mathcal{J}(N_{\sigma}(L))$ where $N_{\sigma}(L)$ is the negative part of the Zariski decomposition.  The diminished ideal then is the natural extension to the pseudo-effective boundary: $\mathcal{J}_{\sigma}(L) = \mathcal{J}(N_{\sigma}(L))$ for every pseudo-effective $L$.  This is a special case of Proposition \ref{zardecom} (2).
\end{exmple}

\begin{lem} 
Let $L$ be a pseudo-effective $\mathbb{R}$-divisor.  We have an equality of sets
\begin{equation*}
\bigcup_{m \in \mathbb{Z}_{> 0}} V(\mathcal{J}_{\sigma}(mL)) = \mathbf{B}_{-}(L).
\end{equation*}
\end{lem}

\begin{proof}
Lemma \ref{perturbedgivesdiminishedlocus} shows that
\begin{equation*}
 \bigcup_{m \in \mathbb{Z}_{> 0}} V(\mathcal{J}_{-}(mL)) = \mathbf{B}_{-}(L).
\end{equation*}
The containments $\mathcal{J}_{\sigma}((m+1)L) \subset \mathcal{J}_{-}((m+1)L) \subset \mathcal{J}_{\sigma}(mL)$ yield the conclusion.
\end{proof}

An immediate consequence of the birational transformation rule for $\mathcal{J}_{-}(L)$ (Corollary \ref{perturbedbirational}) is:

\begin{lem}
Let $L$ be a pseudo-effective divisor.  Suppose that $\phi: Y \to X$ is a birational map from a smooth variety $Y$.  Then
\begin{equation*}
\mathcal{J}_{\sigma}(L) = \phi_{*}(\mathcal{O}_{Y}(K_{Y/X}) \otimes \mathcal{J}_{\sigma}(\phi^{*}L)).
\end{equation*}
\end{lem}

When $L$ has a Zariski decomposition the diminished ideal is particularly easy to understand.

\begin{prop} \label{zardecom}
Let $L$ be a pseudo-effective divisor.
\begin{enumerate}
\item For every prime divisor $\Gamma$ on $X$
\begin{equation*}
\ord_{\Gamma}(\mathcal{J}_{\sigma}(L)) = \lfloor \sigma_{\Gamma}(L) \rfloor.
\end{equation*}
\item Suppose that $L$ has a Zariski decomposition, i.e.~there is a birational map $\phi: Y \to X$ from a smooth variety $Y$ such that $P_{\sigma}(\phi^{*}L)$ is nef.  Then
\begin{equation*}
\mathcal{J}_{\sigma}(L) = \phi_{*}(\mathcal{O}_{Y}(K_{Y/X}) \otimes \mathcal{J}(N_{\sigma}(\phi^{*}L))).
\end{equation*}
\end{enumerate}
\end{prop}

\begin{proof}
(1) Fix an ample divisor $A$ and an $\epsilon > 0$.  Note that
\begin{equation*}
N_{\sigma}(L) \leq (1+\epsilon)N_{\sigma}(L) = N_{\sigma}((1+\epsilon)L)
\end{equation*}
with equality only when $N_{\sigma}(L)=0$.  Since the coefficients of $N_{\sigma}$ vary continuously upon perturbing by an ample divisor, for sufficiently small $\delta$ we have $N_{\sigma}(L) \leq  N_{\sigma}((1+\epsilon)L + \delta A)$.  In particular, for sufficiently small $\epsilon$ and $\delta$ we may ensure that
\begin{align*}
\lfloor \sigma_{\Gamma}(L) \rfloor & = \lfloor \sigma_{\Gamma}((1+\epsilon)L + \delta A) \rfloor.
\end{align*}
Choose $\epsilon$ and $\delta$ so that $\mathcal{J}_{\sigma}(L) = \mathcal{J}(\Vert (1+\epsilon)L + \delta A \Vert)$.  Since $(1+\epsilon)L + \delta A$ is big, Theorem \ref{bigconefunction} shows that
\begin{equation*}
v_{\Gamma}(\Vert (1+\epsilon)L + \delta A \Vert) = \sigma_{\Gamma}((1+\epsilon)L + \delta A).
\end{equation*}
Using valuations to compute the asymptotic multiplier ideal as in Theorem \ref{qvaluativedescription}, we have
\begin{align*}
\ord_{\Gamma}(\mathcal{J}(\Vert (1+\epsilon)L + \delta A \Vert)) & = \lfloor v_{\Gamma}(\Vert (1+\epsilon)L + \delta A \Vert) \rfloor \\
& = \lfloor \sigma_{\Gamma}((1+\epsilon)L + \delta A) \rfloor \\
& = \lfloor \sigma_{\Gamma}(L) \rfloor.
\end{align*}


(2) Let $\psi: W \to Y$ be a log resolution of $N_{\sigma}(\phi^{*}L)$.  Note that $P_{\sigma}(\psi^{*}\phi^{*}L)$ is still nef so that $N_{\sigma}(\psi^{*}\phi^{*}L) = \psi^{*}N_{\sigma}(\phi^{*}L)$.  Applying the birational transformation rule, we may replace $X$ by $W$ so that $P_{\sigma}(L)$ is nef and $N_{\sigma}(L)$ has simple normal crossing support.

\cite{lazarsfeld04} Lemma 9.2.19 shows that  $\mathcal{J}(N_{\sigma}(L)) = \mathcal{O}_{X}(-\lfloor N_{\sigma}(L) \rfloor)$.  In particular $\mathcal{J}(N_{\sigma}(L)) \supset \mathcal{J}_{\sigma}(L)$ by (1).  Conversely, for any ample divisor $A$ we can write
\begin{equation*}
(1+\epsilon)L + A = ((1+\epsilon)P_{\sigma}(L) + A) + (1+\epsilon)N_{\sigma}(L)
\end{equation*}
where the first term is ample.  Thus $\mathcal{J}_{\sigma}(L) \supset \mathcal{J}((1+\epsilon)N_{\sigma}(L))$ for any $\epsilon > 0$.  When $\epsilon$ is sufficiently small $\mathcal{J}((1+\epsilon)N_{\sigma}(L)) = \mathcal{J}(N_{\sigma}(L))$, finishing the proof.
\end{proof}



\begin{proof}[Proof of Theorem \ref{mostlypsefcontainment} (2):]
Suppose that $L$ has a Zariski decomposition.  By the arguments of the proof of Theorem \ref{idealcomparison}, for any birational map $\phi: Y \to X$ we have
\begin{equation*}
\mathcal{J}(h_{min}) \subset \phi_{*}(\mathcal{O}_{Y}(K_{Y/X}) \otimes \mathcal{J}(N_{\sigma}(\phi^{*}L))).
\end{equation*}
Proposition \ref{zardecom} finishes the proof.
\end{proof}

\subsection{Diminished ideals and the balanced decomposition}

While the asymptotic multiplier ideal behaves well on the interior of balanced regions, its behavior along the boundaries is more subtle.  Diminished ideals are a useful tool for understanding this behavior.  

\begin{defn}
Let $L$ be a pseudo-effective divisor.  We say that $L$ is weakly balanced if $\mathcal{J}_{\sigma}(L) = \mathcal{J}_{-}(L)$.
\end{defn}

\begin{rmk}
Corollary \ref{bigsigmaequality} will show that a big divisor is balanced if and only if it is weakly balanced.
\end{rmk}

\begin{prop} \label{perturbedstability}
Let $L$ be a pseudo-effective divisor.  Then
\begin{enumerate}
\item $L + \delta A$ is balanced for any ample divisor $A$ and any sufficiently small $\delta > 0$.
\item  $(1+\epsilon)L$ is weakly balanced for any sufficiently small $\epsilon > 0$.
\end{enumerate}
\end{prop}

\begin{proof} $ $

\begin{enumerate}
\item Choose $\epsilon$ so that $\mathcal{J}_{-}(L) = \mathcal{J}(\Vert L + \epsilon A \Vert)$.  Then any $\delta < \epsilon$ will suffice.

\item Choose $\tau$ sufficiently small so that $\mathcal{J}_{\sigma}(L) = \mathcal{J}_{-}((1+\epsilon)L)$ for every positive $\epsilon < \tau$.  For any such $\epsilon$ we have $\mathcal{J}_{\sigma}((1+\epsilon)L) = \mathcal{J}_{\sigma}(L)$.
\end{enumerate}
\end{proof}

\begin{proof}[Proof of Theorem \ref{mostlypsefcontainment} (1):]
Note that if $L$ is weakly balanced then
\begin{equation*}
\mathcal{J}(h_{min}) \subset \mathcal{J}_{-}(L) = \mathcal{J}_{\sigma}(L)
\end{equation*}
by Theorem \ref{idealcomparison}.  Define $U \subset \mathbb{R}_{>0}$ to be the set of real numbers $\alpha$ such that $\alpha L$ is weakly balanced.  $U$ is open and dense in $\mathbb{R}_{>0}$ by Proposition \ref{perturbedstability}, proving Theorem \ref{mostlypsefcontainment} (1).
\end{proof}

\subsection{Diminished ideals and valuations}

By analogy with the usual log canonical thresholds, we define the $\sigma$-log canonical thresholds as follows.

\begin{defn}
Let $L$ be a pseudo-effective divisor on $X$.  For a non-zero ideal sheaf $\mathfrak{q}$ define
\begin{equation*}
\lct^{\mathfrak{q}}_{\sigma}(L) := \inf_{v \in \Val^{*}(X)} \frac{A(v) + v(\mathfrak{q})}{\sigma_{v}(L)}.
\end{equation*}
\end{defn}

\begin{thrm} \label{sigmalctthm}
Let $L$ be a pseudo-effective divisor on $X$ and let $\mathfrak{q}$ be an ideal sheaf on $X$.
\begin{enumerate}
\item $\mathfrak{q} \subset \mathcal{J}_{-}(L)$ iff $\lct^{\mathfrak{q}}_{\sigma}(L) \geq 1$.
\item $\mathfrak{q} \subset \mathcal{J}_{\sigma}(L)$ iff $\lct^{\mathfrak{q}}_{\sigma}(L) > 1$.
\end{enumerate}
\end{thrm}

\begin{proof}
The first statement is a consequence of Proposition \ref{perturbedvaluation}.  To prove the second, note that $\lct_{\sigma}^{\mathfrak{q}}((1+\epsilon)L) = \frac{1}{1+\epsilon} \lct_{\sigma}^{\mathfrak{q}}(L)$.  Thus (2) follows from (1) by noting that $\mathcal{J}_{-}((1+\epsilon)L) = \mathcal{J}_{\sigma}(L)$ for every sufficiently small $\epsilon$.
\end{proof}


\begin{cor} \label{bigsigmaequality}
Let $L$ be a big divisor.  Then $\mathcal{J}_{\sigma}(L) = \mathcal{J}(\Vert L \Vert)$.
\end{cor}

\begin{proof}
The two ideals have the same valuative description since $\sigma_{v}(L) = v(\Vert L \Vert)$ for every $v \in \Val(X)$.
\end{proof}

\begin{cor} \label{sigmalctlim}
Let $L$ be a pseudo-effective divisor.  Fix an ample divisor $A$.  Then $\lct_{\sigma}^{\mathfrak{q}}(L) = \lim_{\epsilon \to 0^{+}} \lct^{\mathfrak{q}}(L + \epsilon A)$.
\end{cor}

\begin{proof}
Fix $c>0$.  We have $\mathfrak{q} \subset \mathcal{J}_{-}(cL)$ iff $\mathfrak{q} \subset \mathcal{J}(c\Vert L + \epsilon A \Vert)$ for every $\epsilon > 0$.  The latter condition is equivalent to $\lct^{\mathfrak{q}}(\Vert L + \epsilon A \Vert) > c$ for every $\epsilon > 0$.  Theorem \ref{sigmalctthm} (1) shows that the former condition is equivalent to $\lct^{\mathfrak{q}}_{\sigma}(L) \geq c$.  Letting $c$ vary, we obtain the statement.
\end{proof}

In applications, it is crucial to know whether the infimum in the definition of $\lct^{\mathfrak{q}}_{\sigma}$ is achieved by some valuation.  The following theorem gives a partial solution to this problem.

\begin{thrm} \label{diminishedvaluation}
Let $L$ be a pseudo-effective divisor on $X$ and $\mathfrak{q}$ an ideal sheaf on $X$.  Suppose that
\begin{equation*} \label{diminished} \tag{**}
v(\mathfrak{q}) > \sigma_{v}(L) - A(v)
\end{equation*}
for every $v \in \Val^{*}(X)$.  Then $\mathfrak{q} \subset \mathcal{J}_{\sigma}(L)$ if any of the following hold:
\begin{enumerate}
\item $L$ is big.
\item $L$ is weakly balanced.  In this case it suffices to verify \eqref{diminished} for divisorial valuations.
\item $L$ has a Zariski decomposition.  In this case it suffices to verify \eqref{diminished} for divisorial valuations.
\end{enumerate}
\end{thrm}

\begin{proof}
(1) follows from Theorem \ref{qvaluativedescription} and the equality $\sigma_{v}(L) = v(\Vert L \Vert)$.  (2) follows from Proposition \ref{perturbedvaluation}.  (3) follows from Proposition \ref{zardecom}.
\end{proof}

One wonders whether the infimum for $\lct^{\mathfrak{q}}_{\sigma}$ is achieved by an evaluation for every pseudo-effective $L$.  In particular:


\begin{ques}
Let $L$ be a pseudo-effective divisor.  Is there a graded sequence of ideals $\mathfrak{a}_{\bullet}$ such that $\mathcal{J}(\mathfrak{a}_{\bullet}) = \mathcal{J}_{\sigma}(L)$ and $v(\mathfrak{a}_{\bullet}) = \sigma_{v}(L)$ for every valuation $v \in \Val(X)$?
\end{ques}

\begin{rmk}
When $X$ is a surface Theorem \ref{diminishedvaluation} shows that $\mathcal{J}_{\sigma}(L)$ is determined by divisorial valuations.  It would be interesting to see a (higher-dimensional) example where this does not hold.
\end{rmk}

\subsection{Abundant Divisors}

Abundant divisors, introduced by \cite{nakayama04} and \cite{bdpp04}, form a class of pseudo-effective divisors with particularly nice asymptotic behavior.  In this section we will show that abundance can be detected using multiplier ideals on birational models of $X$.  Let $\nu(L)$ denote the numerical dimension as defined by \cite{nakayama04} and \cite{bdpp04}.

\begin{thrm}[\cite{lehmann10}, Theorem 6.1] \label{abundancedefn}
Let $X$ be a smooth variety and $L$ be a pseudo-effective divisor.  The following conditions are equivalent:
\begin{enumerate}
\item $\kappa(L) = \nu(L)$.
\item There is a birational map $\phi: Y \to X$ and a morphism $f: Y \to Z$ such that $P_{\sigma}(f^{*}L) \sim_{\mathbb{Q}} f^{*}B$ for some big divisor $B$ on $Z$.
\item Let $\phi: Y \to X$ denote a birational map resolving the Iitaka fibration $f: Y \to Z$.  Then $\nu(\phi^{*}L|_{F}) = 0$ for a general fiber $F$ of $f$.
\end{enumerate}
If these conditions hold for $L$, we say that $L$ is abundant.
\end{thrm}

Note in particular that every big divisor is abundant.  Furthermore, it is clear that $L$ is abundant if and only if $P_{\sigma}(L)$ is abundant.  One consequence of abundance is the following equality:

\begin{prop}[\cite{lehmann10}, Proposition 6.4] \label{abundancevaluationequality}
Let $X$ be a smooth variety and let $L$ be an abundant divisor.  Then
\begin{equation*}
\sigma_{v}(L) = v(\Vert L \Vert)
\end{equation*}
for every divisorial valuation $v$.
\end{prop}

Conversely, we show that if $L$ is not abundant then there is some asymptotic divisorial valuation $v(\Vert L \Vert)$ that disagrees with $\sigma_{v}(L)$.  The key is to focus on components of the stable base locus of $L$ that dominate the base of the Iitaka fibration.  This idea has been used before, for example in \cite{takayama03}.

\begin{prop} \label{iitakafibrationdomination}
Let $X$ be a smooth variety and $L$ an $\mathbb{R}$-divisor with $\kappa(L) \geq 0$.  Let $\phi: Y \to X$ denote a resolution of the Iitaka fibration $f: Y \to Z$.  Choose an effective divisor $D \sim_{\mathbb{Q}} L$ and split $\phi^{*}D$ into horizontal and vertical components as
\begin{equation*}
\phi^{*}D = D_{hor} + D_{ver} = \sum a_{i}E_{i} + \sum b_{i}F_{i}.
\end{equation*}
For every horizontal component $E_{i}$ we have
\begin{equation*}
v_{\phi(E_{i})}(\Vert mL \Vert) = ma_{i}.
\end{equation*}
\end{prop}

\begin{proof}
Let $F$ be a very general fiber of $f$ so that $\kappa(D|_{F}) = 0$.  If $D' \sim_{\mathbb{Q}} D$ is effective then $D'|_{F} = D|_{F} = \sum a_{i}E_{i}|_{F}$.  We immediately obtain $v_{E_{i}}(\Vert m\phi^{*}L \Vert) = ma_{i}$.
\end{proof}

\begin{proof}[Proof of Theorem \ref{abundantequality}:]
Suppose $L$ is abundant.  Choose a rational $\epsilon > 0$ sufficiently small so that $(1+\epsilon)L$ is weakly balanced.  We may also suppose that $\mathcal{J}_{\sigma}((1+\epsilon)L) = \mathcal{J}_{\sigma}(L)$ and $\mathcal{J}(\Vert (1+\epsilon) L \Vert) = \mathcal{J}(\Vert L \Vert)$.  Combining the valuative description of $\mathcal{J}_{\sigma}(L)$ for weakly balanced divisors in Theorem \ref{diminishedvaluation} (2) with Proposition \ref{abundancevaluationequality} we find
\begin{equation*}
\mathcal{J}_{\sigma}(L) = \mathcal{J}_{\sigma}((1+\epsilon)L) = \mathcal{J}(\Vert (1+\epsilon) L\Vert) = \mathcal{J}(\Vert L \Vert).
\end{equation*}

Conversely, as the $\kappa(L) = -\infty$ case is immediate, we restrict our attention to $\kappa(L) \geq 0$.  Let $\phi: Y \to X$ be a resolution of the Iitaka fibration $f: Y \to Z$.  We may assume for convenience that $L \geq 0$.  If $L$ is not abundant, then $P_{\sigma}(\phi^{*}L)$ is also not abundant, and characterization (3) of Theorem \ref{abundancedefn} applied to $P_{\sigma}(\phi^{*}L)$ implies that there is some component $E$ of $P_{\sigma}(\phi^{*}L)$ such that $f(E)=Z$.  In particular
\begin{equation*}
v_{E}(\Vert \phi^{*}L \Vert) > v_{E}(N_{\sigma}(\phi^{*}L)) = \sigma_{E}(\phi^{*}L)
\end{equation*}
by Proposition \ref{iitakafibrationdomination}.   Choose $m$ sufficiently large so that for some integer N we have
\begin{equation*}
mv_{E}(\Vert \phi^{*}L \Vert) > N + 1 > m\sigma_{E}(\phi^{*}L)+1.
\end{equation*}
Choose $\epsilon$ sufficiently small so that $(1+\epsilon)\phi^{*}L$ is weakly balanced and also $\epsilon m \sigma_{E}(\phi^{*}L) < 1$.  We have $I_{E}^{N} \subset \mathcal{J}_{\sigma}((1+\epsilon)m\phi^{*}L) \subset \mathcal{J}_{\sigma}(m\phi^{*}L)$ by Proposition \ref{zardecom}.  But $I_{E}^{N} \not \subset \mathcal{J}(\Vert m\phi^{*} L \Vert)$ by Theorem \ref{qvaluativedescription}.  Thus $\mathcal{J}(\Vert m\phi^{*}L \Vert) \subsetneq \mathcal{J}_{\sigma}(m\phi^{*}L)$.
\end{proof}

In fact, we have proven that Theorem \ref{abundantequality} can be tested on any resolution of the Iitaka fibration.

\section{Examples} \label{examplessection}

We conclude by mentioning several classes of variety for which $\mathcal{J}(h_{min})$ is determined by Theorems \ref{idealcomparison} and \ref{abundantequality}: Mori Dream Spaces,  minimal surfaces of Kodaira dimension $0$, and most minimal ruled surfaces.

When these theorems do not apply there is currently no geometric method for understanding $\mathcal{J}(h_{min})$.  As in \cite{dps01} Example 2.5, lifting theorems can occasionally yield useful information: since the analytic multiplier ideal controls the ability to lift sections, it is sometimes possible to give an upper bound for $\mathcal{J}(h_{min})$ by showing some sections do not lift.

\subsection{Mori Dream Spaces}

Suppose that $L$ is a pseudo-effective divisor on a Mori Dream Space $X$ (such as a smooth Fano or toric variety).  Then $L$ is abundant and has a Zariski decomposition.  Theorems \ref{mostlypsefcontainment} and \ref{abundantequality} show that $\mathcal{J}(h_{min}) = \mathcal{J}(\Vert L \Vert)$.

\subsection{Surfaces}

Every divisor on a surface has a Zariski decomposition.  Theorems \ref{mostlypsefcontainment} and \ref{abundantequality} yield

\begin{thrm} \label{surfacetheorem}
Let $S$ be a smooth surface and let $L$ be a pseudo-effective divisor on $S$.  Then
\begin{equation*}
\mathcal{J}(\Vert L \Vert) \subset \mathcal{J}(h_{min}) \subset \mathcal{J}_{\sigma}(L).
\end{equation*}
If $L$ is abundant then all three ideals coincide.
\end{thrm}

(One could also appeal to the multiplier ideal version of the Openness Conjecture for surfaces which was established in \cite{fj04}.)  We will use this result to analyze minimal surfaces with $\kappa(S) = -\infty,0$.

\subsubsection{Ruled surfaces}
Let $S$ be a minimal ruled surface.  The N\'eron-Severi space of $S$ is two dimensional.  One extremal ray of $\overline{NE}^{1}(S)$ is generated by a fiber of the ruling; we will let $L$ denote a divisor generating the other extremal ray.  Since every divisor besides $L$ is abundant, it only remains to consider the value of $\mathcal{J}(h_{min})$ for $L$ itself.  If $L$ is not nef, then $P_{\sigma}(L)$, and hence $L$, must be abundant.

The only interesting situation is when every pseudo-effective divisor is nef, or equivalently, when $\mathcal{E}$ is semistable.  Many of the remaining cases are settled by \cite{eckl04} 4.2.  The difficult case seems to be when $\mathcal{E}$ is strictly semistable but not split.

\subsubsection{Minimal surfaces with $\kappa(S) = 0$}
We next describe the possible behaviors of $\mathcal{J}(h_{min})$ on minimal surfaces of Kodaira dimension $0$.  For a pseudo-effective Cartier divisor $L$ on such a surface we have
\begin{equation*}
\mathcal{J}(h_{min}) = \mathcal{J}(\Vert L' \Vert)
\end{equation*}
where $L'$ is any effective divisor numerically equivalent to $L$.  Since $\mathcal{J}(h_{min})$ is a numerical invariant it suffices to prove the following:

\begin{lem}
Let $L$ be a Cartier divisor on a minimal surface $S$ with $\kappa(S) = 0$.  Then $L$ is numerically equivalent to an abundant divisor.
\end{lem}

\begin{proof}
Suppose that $f: S' \to S$ is a finite map.  Then $L$ is numerically equivalent to an abundant divisor if and only if $f^{*}L$ is.  In this way we reduce to the case where $S$ is a K3 surface or an abelian surface.  We now refer to two classical facts:
\begin{enumerate}
\item A pseudo-effective divisor on a K3 surface has non-negative Iitaka dimension.
\item A pseudo-effective divisor on an abelian surface is algebraically equivalent to an effective divisor.
\end{enumerate}
We conclude by the Abundance Theorem for surfaces.
\end{proof}

\nocite{*}
\bibliographystyle{amsalpha}
\bibliography{minsing}

\end{document}